\documentclass{amsart}
\usepackage{amsmath, amssymb, amscd}

\newcommand{\B}[1]{\mathbb{#1}}
\newcommand{\C}[1]{\mathcal{#1}}

\newcommand{\xra}[1]{\xrightarrow{#1}}
\newcommand{\xla}[1]{\xleftarrow{#1}}

\newcommand{\dglmod}[1]{{#1}\text{-\bf dgmod}}
\newcommand{\rmod}[1]{\text{{\bf mod}-}{#1}}

\newtheorem{thm}{Theorem}[section]

\newtheorem{lem}[thm]{Lemma}
\newtheorem{prop}[thm]{Proposition}

\theoremstyle{definition}
\newtheorem{defn}[thm]{Definition}
\newtheorem{rem}[thm]{Remark}

\numberwithin{equation}{section}

\author{Atabey Kaygun}
\email{kaygun@math.ohio-state.edu}
\address{Department of Mathematics, The Ohio State University, Columbus 43210, Ohio USA}
\title{Products in Hopf-cyclic cohomology}
\begin{document}
\maketitle

\section{Introduction}

It is a truism to say that the study of Hopf algebras provides an
unified framework for a wide range of algebras such as group rings,
enveloping algebras of Lie algebras and quantum groups.  Beyond this
unified framework, Hopf algebras also provide the true symmetries of
noncommutative spaces.  Yet, as it is observed in
\cite{Khalkhali:IntroHopfCyclicCohomology} and
\cite{Kaygun:UniversalHopfCyclicTheory} the notion of `symmetry of a
noncommutative space' is not as straightforward as its classical
counterpart in that there are fundamentally different types of
`equivariances' for noncommutative spaces.  In this article, we will
investigate cohomological consequences of these different types of
Hopf equivariances under various scenarios.

Hopf-cyclic (co)homology and equivariant cyclic cohomology of module
(co)algebras, in their most general form as defined
in~\cite{Kaygun:BialgebraCyclicK} and
\cite{Khalkhali:HopfCyclicHomology}, can now be described via
appropriate bivariant cohomology theories~\cite{Kaygun:BivariantHopf}.
These bivariant cohomology theories provide us the right vocabulary
and the right kind of tools to investigate various kinds of products
and pairings in the Hopf-cyclic and equivariant cyclic settings.  We
already established in~\cite{Kaygun:BivariantHopf} that equivariant
cyclic (co)homology groups of $H$-module (co)algebras are graded
modules over the graded algebra ${\rm Ext}^*_H(k,k)$ where $H$ is the
underlying Hopf algebra.  Not only is this the right kind of
equivariance for an equivariant cohomology theory but also is a very
useful computational tool in obtaining new equivariant cyclic classes
from old ones by using the action.  However, Hopf-cyclic cohomology
lacks such an action precisely because it is, by design, oblivious to
the cohomology of the underlying Hopf algebra viewed as an algebra or
a coalgebra depending on the type of the symmetry at hand.  To improve
the situation, in Theorem~\ref{HopfCupProduct} we develop a pairing,
similar to one of the cup products obtained in
~\cite{Khalkhali:CupProducts}, of the form
\begin{align*}
  \smile\colon HC^p_{\rm Hopf}(C,M)\otimes HC^q_{\rm Hopf}(A,M)\to & HC^{p+q}(A) 
\end{align*}
for an $H$-module coalgebra $C$ acting equivariantly on a $H$-module
algebra $A$ (Definition~\ref{EquivariantAction}) where $M$ is an
arbitrary coefficient module/comodule.  The most striking corollary to
this pairing is Theorem~\ref{Agreement} where we recover
Connes-Moscovici characteristic map \cite[Section VIII, Proposition
1]{ConnesMoscovici:HopfCyclicCohomology} by letting $q=0$, $C=H$ and
letting $M = k_{(\sigma,\delta)}$ a coefficient module coming from a
modular pair in involution.  The ideas instrumental in constructing
this pairing are the derived functor interpretation of Hopf-cyclic and
equivariant cyclic cohomology~\cite{Kaygun:BivariantHopf}, and the
Yoneda interpretation of Ext-groups~\cite{Yoneda:ExtGroups}.  We use
the same ideas with little modification in
Theorem~\ref{EquivariantProduct} to obtain a completely new pairing in
equivariant cyclic cohomology of the form
\begin{align*}
    \smile\colon HC^p_H(C,M)\otimes HC^q_H(A,M)\to & 
          \bigoplus_{r=0}^{p+q} HC^{p+q-r}(A)\otimes {\rm Ext}_H^r(k,k)
\end{align*}
where, as before, $C$ is an $H$-module coalgebra acting equivariantly
on an $H$-module algebra $A$, and $M$ is an arbitrary coefficient
module/comodule.

Once we couple equivariant/Hopf-cyclic cohomology and the Yoneda
product, we obtain other types of products and pairings by imposing
various conditions on the underlying Hopf algebra, the coefficient
modules and (co)module (co)algebras involved
\begin{align*}
  HC^p_{\rm Hopf}(A,M)\otimes HC^q_{\rm Hopf}(A',M')\to & HC^{p+q}_{\rm Hopf}(A\otimes A',M\Box^H M')
          & \text{Theorem }\ref{CocommutativeExternalProduct}\\
  HC^p_H(A,M)\otimes HC^q_H(A',M')\to & HC^{p+q}_H(A\otimes A',M\Box^H M')
          & \text{Theorem }\ref{CocommutativeExternalProduct}\\
  HC^{\vee,p}_{\rm Hopf}(Z,M)\otimes HC^{\vee,q}_{\rm Hopf}(Z',M')\to &
     HC^{\vee,p+q}_{\rm Hopf}(Z\otimes Z', M\otimes_H M')
          & \text{Theorem }\ref{CommutativeExternalProduct}\\
  HC^p_{\rm Hopf}(A,M)\otimes HC^q_{\rm Hopf}(B,M)\to & HC^{p+q}(A\rtimes B)  
          & \text{Theorem }\ref{CupProductCrossedProduct}\\
  HC^p(Z\ltimes C)\otimes HC_{\rm Hopf}^{\vee,q}(Z,M)\to & HC_{\rm Hopf}^{p+q}(C,M) 
          & \text{Theorem }\ref{CrossedProductCoalgebra}
\end{align*}
where $HC^*$, $HC^*_{\rm Hopf}$, $HC^*_H$ and $HC^{\vee,*}_{\rm Hopf}$
are ordinary cyclic, Hopf-cyclic, equivariant cyclic and dual Hopf
cyclic cohomology functors respectively.

The last section of our paper is devoted to ramifications of an
interesting technical problem in cyclic cohomology which manifests
itself here in the context of pairings in Hopf-cyclic (co)homology.
It is clear by now that there are essentially two different
homotopical frameworks for the category cyclic modules: (i) Connes'
derived category of cyclic modules~\cite{Connes:ExtFunctors} and (ii)
Cuntz-Quillen formalism of homotopy category of towers of super
complexes~\cite{CuntzQuillen:NonsingularityII} which is equivalent to
the derived category of mixed complexes~\cite{Kassel:MixedComplexes}
and the derived category of
$S$-modules~\cite{Kassel:BivariantChernCharacter} by
\cite{Quillen:CyclicHomologyType}.  One can see the difference in the
simple fact that for a cyclic module $X_\bullet$, the derived functors
${\rm Ext}^*_{\Lambda}(k_\bullet^\vee,X_\bullet)$ in the category of
cyclic modules compute the dual cyclic homology of $X_\bullet$ while
the derived functors ${\rm\bf Ext}^*_\C{M}(\C{B}_*(k_\bullet^\vee),
\C{B}_*(X_\bullet))$ in the category of mixed complexes compute the
negative cyclic homology of $X_\bullet$ \cite[Theorem
2.3]{JonesKassel:BivariantCyclicTheory}.  Here $k_\bullet$ is the
cocyclic module of the ground field viewed as a coalgebra, the dual
cyclic homology of $X_\bullet$ is the cyclic cohomology of
$X_\bullet^\vee$, the cyclic dual of a (co)cyclic module $X_\bullet$
defined by using Connes' duality functor~\cite{Connes:ExtFunctors},
and $\C{B}_*(Z_\bullet)$ is the mixed complex of a (co)cyclic module
$Z_\bullet$.  The main result of this last section is
Theorem~\ref{AnalogousPairings} where we prove replacing the derived
category cyclic modules by the derived category of mixed complexes, or
$S$-modules, or the homotopy category of towers of super complexes in
Theorem~\ref{HopfCupProduct}, Theorem~\ref{EquivariantProduct} and
Theorem~\ref{CocommutativeExternalProduct} will not change the
pairings we already defined.

In this article $k$ will denote an arbitrary field.  We make no
assumption about the characteristic of $k$.  We will use $H$ to denote
a bialgebra, or a Hopf algebra with an invertible antipode over $k$,
whenever necessary.  All tensor products, unless otherwise explicitly
stated, are over $k$.

\vspace{3mm} 

\noindent{\bf Acknowledgements} \\ This article is conceived during my
stay in Warsaw in November and December of 2006 as a part of Transfer
of Knowledge Programme in Noncommutative Geometry and Quantum Groups.
I was partially supported by Marie Curie Fellowship
MKTD-CT-2004-509794.  I would like to thank the Department of Physics
of Warsaw University and the Banach Center for their invitation, warm
hospitality and support.  Especially, I would like to thank Piotr M.
Hajac for suggesting that I gave a talk about cup products and
characteristic map in Hopf-cyclic cohomology in his Noncommutative
Geometry Seminar, and our ensuing discussions on Hopf-cyclic
cohomology.

\section{Equivariant actions of coalgebras on algebras}

In this section, $A$ will denote a unital associative left $H$-module
algebra and $C$ will denote a counital coassociative left $H$-module
coalgebra.  Explicitly, one has
\begin{align*}
  h(a_1a_2) =  (h_{(1)} a_1) (h_{(2)} a_2) \quad\text{ and }\quad
  (hc)_{(1)}\otimes (hc)_{(2)} = h_{(1)} c_{(1)}\otimes h_{(2)} c_{(2)}
\end{align*}
for any $a_1,a_2\in A$, $c\in C$ and $h\in H$.  We also assume
\begin{align*}
  h(1_A) = & \varepsilon(h) 1_A \quad\text{ and }\quad
  \varepsilon(hc) = \varepsilon(h)\varepsilon(c)
\end{align*}
for any $c\in C$ and $h\in H$.  We will use $M$ to denote an arbitrary
$H$-module/comodule with no assumption on the interaction between the
$H$-module and $H$-comodule structures on $M$.

Given two morphisms $f_1,f_2\in {\rm Hom}_k(C,A)$ we define their
convolution product as
\begin{align*}
    (f_1* f_2)(c) := f_1(c_{(1)}) f_2(c_{(2)})
\end{align*}
for any $c\in C$.  This binary operation on ${\rm Hom}_k(C,A)$ is an
associative product.  The unit for this algebra is $\eta(c):=
\varepsilon(c) 1_A$.  The proof of the following lemma is routine.

\begin{lem}\label{Pairing}
  There exists a morphism of algebras of the form $\alpha\colon
  A\to{\rm Hom}_k(C,A)$ if and only if one has a pairing $\phi\colon
  C\otimes A\to A$ which satisfies
  \begin{align*}
    \phi(c,a_1 a_2) = \phi(c_{(1)},a_1) \phi(c_{(2)},a_2)
    \quad\text{ and }\quad
    \phi(c,1) = \varepsilon(c) 1_A
  \end{align*}
  for any $c\in C$, $a,a_1,a_2\in A$ and $h\in H$.
\end{lem}


\begin{defn}\label{EquivariantAction}
If one has a pairing between a module coalgebra $C$ and a module
algebra $A$ as described in Lemma~\ref{Pairing} then $C$ is said to
act on $A$.  Such an action is going to be called equivariant if one
also has
\begin{align*}
  h\phi(c,a) = \phi(h c, a)
\end{align*}
for any $a\in A$, $h\in H$ and $c\in C$.
\end{defn}

One can observe that an $H$-module coalgebra $C$ acts on an $H$-module
algebra $A$ equivariantly if and only if the canonical morphism of
algebras $\alpha\colon A\to {\rm Hom}_k(C,A)$ factors through the
inclusion ${\rm Hom}_H(C,A)\subseteq {\rm Hom}_k(C,A)$
\cite{Khalkhali:CupProducts}.

\begin{defn}
  If $X_\bullet$ is a (para-)cocyclic $k$-module and $Y_\bullet$ is a
  (para-)cyclic $k$-module then the graded module 
  \[ diag{\rm  Hom}_k(X_\bullet,Y_\bullet):= \bigoplus_n {\rm Hom}_k(X_n,Y_n) \]
  carries a a (para-)cyclic module structure defined as
  \begin{align*}
    (\partial_j f)(x_{n-1}) := & \partial^Y_j f(\partial^X_j(x_{n-1})) \quad 0\leq j\leq n\\
    (\sigma_j f_n)(x_{n+1}) := & \sigma^Y_j f(\sigma^X_j(x_{n+1})) \quad 0\leq j\leq n\\
    (\tau_n f_n)(x_n) := & \tau_{n,Y} f (\tau_{n,X}(x_n))
  \end{align*}
  for any $f\in {\rm Hom}_k(X_n,Y_n)$, $x_i\in X_i$ for $i=n-1,n,n+1$.
\end{defn}

We will use $Cyc_\bullet(X)$ to denote the classical cyclic $k$-module
of an (co)associative (co)unital $k$-(co)algebra $X$.  Also, we will
use $k_\bullet$ to denote $Cyc_\bullet(k^c)$ the {\em cocyclic}
$k$-module of the ground field viewed as a coalgebra and
$k_\bullet^\vee$ to denote $Cyc_\bullet(k)$ the {\em cyclic}
$k$-module of the ground field $k$ viewed as an algebra.  Note that in
this case $k_\bullet^\vee$ is actually the cyclic dual of the cocyclic
$k$-module $k_\bullet$ in the sense of \cite[Lemme
1]{Connes:ExtFunctors}.  In general, if $X_\bullet$ is an arbitrary
(co)cyclic $k$-module, $X_\bullet^\vee$ will denote its cyclic dual.
$T_\bullet(C,M)$ and $T_\bullet(A,M)$ are going to denote the
para-(co)cyclic complex (also referred as ``the cover complex'') of
the $H$-module coalgebra $C$ and $H$-module algebra $A$ with
coefficients in a $H$-module/comodule $M$
respectively~\cite{Kaygun:BivariantHopf}.

Here we recall the para-(co)cyclic structure morphisms on both
$T_\bullet(C,M)$ and $T_\bullet(A,M)$ from
\cite{Kaygun:BivariantHopf}.  The modules are defined as
\begin{align*}
 T_n(C,M):= C^{\otimes n+1}\otimes M 
 \quad\text{ and }\quad
 T_n(A,M):= A^{\otimes n+1}\otimes M
\end{align*}
We define the structure morphisms on $T_\bullet(C,M)$ by
\begin{align*}
 \partial_0(c^0\otimes\cdots\otimes c^n\otimes m)
  := & c^0_{(1)}\otimes c^0_{(2)}\otimes c^1\otimes\cdots\otimes c^n\otimes m\\
\sigma_0(c^0\otimes\cdots\otimes c^n\otimes m)
  := & c^0\otimes \varepsilon(c^1)\otimes c^2\cdots\otimes c^n\otimes m\\
\tau_n(c^0\otimes\cdots\otimes c^n\otimes m)
  := & c^1\otimes\cdots\otimes c^n\otimes m_{(-1)}c^0\otimes m_{(0)}
\end{align*}
Then we define $\partial_j := \tau_{n+1}^{-j}\partial_0\tau_n^j$ and
$\sigma_i := \tau_{n-1}^{-j}\sigma_0\tau_n^j$.  Similarly, we let
\begin{align*}
  \partial_0(a_0\otimes\cdots\otimes a_n\otimes m)
   := & a_0 a_1\otimes a_2\otimes a_n\otimes m\\
  \sigma_0(a_0\otimes\cdots\otimes a_n\otimes m)
   := & a_0\otimes 1_A\otimes a_1\otimes\cdots\otimes a_n\otimes m\\
  \tau_n(a_0\otimes\cdots\otimes a_n\otimes m)
   := & S^{-1}(m_{(-1)}) a_n\otimes a_0\otimes\cdots\otimes 
        a_{n-1}\otimes m_{(0)}
\end{align*}
Then we define $\partial_j := \tau_{n-1}^j\partial_0\tau_n^{-j}$
and $\sigma_j := \tau_{n+1}^j\sigma_0\tau_n^{-j}$.

\begin{prop}\label{Lift}
  Assume $C$ acts on $A$ equivariantly and let $M$ be an arbitrary
  $H$-module/comodule.  Let us define
  \begin{align}\label{Characteristic}
    \alpha_n(a_0\otimes\cdots\otimes a_n)(c^0\otimes\cdots\otimes c^n\otimes m)
    := & \phi(c^0,a_0)\otimes\cdots\otimes \phi(c^n,a_n)\otimes m
  \end{align}
  for any $n\geq 0$, $m\in M$, $a_i\in A$, $c^i\in C$ for
  $i=0,\ldots,n$.  Then $\alpha_\bullet$ defines a morphism of
  para-cyclic $k$-modules of the form
  \begin{align*}
    \alpha_\bullet\colon Cyc_\bullet(A) \to 
    diag{\rm Hom}_H(T_\bullet(C,M), T_\bullet(A,M))
  \end{align*}
\end{prop}

\begin{proof}
  It is easy to observe that $\alpha_n$ is $H$-linear since the action
  of $C$ on $A$ is equivariant.  Let us check if $\alpha_\bullet$ now
  defines a morphism of para-cyclic modules:
  \begin{align}
    \alpha_{n-1}(\partial^A_0(a_0\otimes\cdots\otimes a_n)) 
      & (c^0\otimes\cdots\otimes c^{n-1}\otimes m)\nonumber\\
    = & \phi(c^0,a_0a_1)\otimes \phi(c^1,a_2)\otimes\cdots\otimes
           \phi(c^{n-1},a_n)\otimes m\nonumber\\
    = & \phi(c^0_{(1)},a_0)\phi(c^0_{(2)},a_1)\otimes \phi(c^1,a_2)
           \otimes\cdots\otimes\phi(c^{n-1},a_n)\otimes m\label{faces}\\
    = & \partial^{(A,M)}_0\alpha_n(a_0\otimes\cdots\otimes a_n)
        \partial^{(C,M)}_0(c_0\otimes\cdots\otimes c^{n-1}\otimes m)\nonumber
  \end{align}
  \begin{align}
    \alpha_{n+1}(\sigma^A_0(a_0\otimes\cdots\otimes a_n)) 
      & (c^0\otimes\cdots\otimes c^{n+1}\otimes m)\nonumber\\
     = & \phi(c^0,a_0)\otimes \varepsilon(c_1)1_A\otimes\phi(c^2,a_1)
         \otimes\cdots\otimes \phi(c^{n+1},a_n)\otimes m\label{degeneracy}\\
     = & \sigma^{(A,M)}_0\alpha_{n+1}(a_0\otimes\cdots\otimes a_n)
       \sigma^{(C,M)}_0(c^0\otimes\cdots\otimes c^{n+1}\otimes m)\nonumber
  \end{align}
  \begin{align}
    \alpha_n(\tau_n(a_0\otimes\cdots\otimes a_n))
      & (c_0\otimes\cdots\otimes c_n\otimes m)\nonumber\\
     = & \phi(c^0,a_n)\otimes \phi(c^1,a_0)\otimes\cdots
         \otimes\phi(c^n,a_{n-1})\otimes m\nonumber\\
     = & S^{-1}(m_{(-1)}) \phi(m_{(-2)} c^0,a_n)\otimes \phi(c^1,a_0)
         \otimes\cdots\otimes\phi(c^n,a_{n-1})\otimes m_{(0)}\label{cyclic}\\
     = & \tau_{n,(A,M)}\big(
         \phi(c^1,a_0)\otimes\cdots\otimes \phi(c^n,a_{n-1})
          \otimes \phi(m_{(-1)} c^0,a_n)\otimes m_{(0)}\big)\nonumber\\
     = & \tau_{n,(A,M)}\alpha_n(a_0\otimes\cdots\otimes a_n)\tau_{(C,M)}
         (c_0\otimes\cdots\otimes c_n\otimes m)\nonumber
  \end{align}
for any $c^i\in C$, $a_i\in A$, $m\in M$.
\end{proof}

\begin{rem}

Now, we will recall few relevant definitions from
\cite{Kaygun:BivariantHopf}.

Let $J_\bullet(C,M)$ be the smallest para-cocyclic $k$-submodule and
graded $H$-submodule (but not necessarily the para-cocyclic
$H$-subcomodule) of $T_\bullet(C,M)$ generated by elements of the form
$[L_h,\tau_n^i](\Psi) + (\tau_n^{n+1}-id_n)(\Phi)$ where $\Psi,\Phi\in
T_n(C,M)$, $i\in\B{Z}$ and $L_h$ is the graded $k$-module endomorphism
of $T_\bullet(C,M)$ coming from the left diagonal action of $h\in H$
on $T_n(C,M)$ for each $n\geq 0$.  One can similarly define
$J_\bullet(A,M)$.  

We define $Q_\bullet(C,M) := T_\bullet(C,M)/J_\bullet(C,M)$.  One can
see that $Q_\bullet(C,M)$ is a cocyclic $H$-module.  Similarly
$Q_\bullet(A,M):= T_\bullet(A,M)/J_\bullet(A,M)$ is a cyclic
$H$-module.  This cocyclic (resp.  cyclic) $H$-module is called the
$H$-equivariant cocyclic (resp. cyclic) module of the pair $(C,M)$
(resp. $(A,M)$).  The cyclic cohomology of the (co)cyclic $H$-modules
$Q_\bullet(C,M)$ and $Q_\bullet(A,M)$ will be denoted by $HC_H^*(C,M)$
and $HC_H^*(A,M)$ respectively.

We define $C_\bullet(C,M) := k\otimes_H Q_\bullet(C,M)$.  One can see
that $C_\bullet(C,M)$ is a cocyclic $k$-module.  Similarly
$C_\bullet(A,M):= k\otimes_H Q_\bullet(A,M)$ is a cyclic $k$-module.
This cocyclic (resp.  cyclic) $k$-module is called the Hopf-cocyclic
(resp. Hopf-cyclic) module of the pair $(C,M)$ (resp. $(A,M)$).  The
cyclic cohomology of the (co)cyclic $k$-modules $C_\bullet(C,M)$ and
$C_\bullet(A,M)$ will be denoted by $HC_{\rm Hopf}^*(C,M)$ and
$HC_{\rm Hopf}^*(A,M)$ respectively.

\end{rem}

\begin{lem}~\label{Restriction}
  For any $n\geq 0$ and $a_i\in A$ for $0\leq i\leq n$ the restriction
  of the $H$-linear morphism $\alpha_n(a_0\otimes\cdots\otimes a_n)$
  in ${\rm Hom}_k(T_n(A,M),T_n(C,M))$ to $J_n(C,M)$ is a morphism in
  ${\rm Hom}_H(J_n(C,M),J_n(A,M))$.
\end{lem}

\begin{proof}
  It is sufficient to prove that for every $n\geq 0$, $a_i\in A$ with
  $0\leq i\leq n$ the morphism $\alpha_n(a_0\otimes\cdots\otimes a_n)$
  maps the elements of the form $[L_h,\tau_n^i](\Psi) +
  (\tau_n^{n+1}-id_n)(\Phi)$ to elements of the same form.  We know
  that each $\alpha_n(a_0\otimes\cdots\otimes a_n)$ is $H$-linear.
  So, we can reduce the proof to verification of the following string
  of equalities:
  \begin{align*}
    \alpha_n(a_0\otimes\cdots\otimes a_n) & 
    \tau_{n,(C,M)}(c_0\otimes\cdots\otimes c_n\otimes m)\\
     = & \alpha_n(a_0\otimes\cdots\otimes a_n) 
         (c^1\otimes\cdots\otimes c^n\otimes m_{(-1)}c^0\otimes m_{(0)})\\
     = & \phi(c^1,a_0)\otimes\cdots\otimes\phi(c^n,a_{n-1})
         \otimes m_{(-1)}\phi(c^0,a_n)\otimes m_{(0)}\\
     = & \tau_{n,(A,M)}^{-1}(\phi(c^0,a_n)\otimes\phi(c^1,a_0)\otimes\cdots\otimes \phi(c^n,a_{n-1})\otimes m)
  \end{align*}
  for $m\in M$, $a_i\in A$, $c^i\in C$ with $0\leq i\leq n$.
\end{proof}

\begin{prop}\label{UniversalClass}
  $\alpha_\bullet$ can be extended to morphism of cyclic $k$-modules of the form
  \[ \alpha_\bullet\colon Cyc_\bullet(A)\to diag {\rm Hom}_H(Q_\bullet(C,M),Q_\bullet(A,M)) 
     \quad\text{and}\quad 
     \alpha_\bullet\colon Cyc_\bullet(A)\to diag {\rm Hom}_k(C_\bullet(C,M),C_\bullet(A,M)) 
  \]
\end{prop}

\begin{proof}
  The first part of the statement follows from
  Lemma~\ref{Restriction}.  The second part follows from the fact that
  $k\otimes_H(\ \cdot\ )$ is a functor from the category of left
  $H$-modules to the category of $k$-modules.
\end{proof}

\begin{thm}\label{HopfCupProduct}
  The equivariant action of an $H$-module coalgebra $C$ on an
  $H$-module algebra $A$ induces a pairing of the form
  \begin{align*}
    \smile\colon HC^p_{\rm Hopf}(C,M)\otimes HC^q_{\rm Hopf}(A,M)\to HC^{p+q}(A)
  \end{align*}
  for any $p,q\geq 0$.
\end{thm}

\begin{proof}
  First, we observe that
  \[ HC^p_{\rm Hopf}(C,M) := {\rm Ext}^p_{\Lambda} (k_\bullet, C_\bullet(C,M))
     \quad\text{ and }\quad
     HC^q_{\rm Hopf}(A,M) := {\rm Ext}^q_{\Lambda}(C_\bullet(A,M),k_\bullet^\vee)
  \]
  We will use the Yoneda interpretation of
  Ext-groups~\cite{Yoneda:ExtGroups} as developed in \cite[Chapter
  III]{MacLane:Homology}.  Our approach in part is inspired by the use
  of Yoneda Ext-groups in \cite{Nistor:BivariantChernConnes}.  In this
  approach, one can represent any Hopf-cyclic cohomology class $\xi\in
  HC^p_{\rm Hopf}(k_\bullet,C_\bullet(C,M))$ and $\nu\in HC^q_{\rm
    Hopf}(C_\bullet(A,M),k_\bullet^\vee)$ by exact sequences of
  (co)cyclic $k$-modules
  \begin{align*}
    \xi: &\quad 0\xla{}k_\bullet\xla{}Y^1_\bullet\xla{}\cdots\xla{}Y^p_\bullet\xla{}C_\bullet(C,M)\xla{}0\\
    \nu: &\quad 0\xla{}C_\bullet(A,M)\xla{}Z^1_\bullet\xla{}\cdots\xla{}Z^q_\bullet\xla{}k_\bullet^\vee\xla{}0
  \end{align*}
  Since $k$ is a field, the functor $diag{\rm Hom}_k(\ \cdot\ ,
  C_\bullet(A,M))$ from the category of cocyclic $k$-modules to the
  category of cyclic $k$-modules is exact.  Hence we get an exact
  sequence of the form
  \begin{align*}
    diag{\rm Hom}_k(\xi,C_\bullet(A,M)):\quad 
    0\xla{}  diag{\rm Hom}_k(C_\bullet(C,M),C_\bullet(A,M))\xla{}
            diag{\rm Hom}_k(Y^p_\bullet,C_\bullet(A,M))\xla{}\cdots\\ 
            \xla{} diag{\rm Hom}_k(Y^1_\bullet,C_\bullet(A,M))\xla{}
             C_\bullet(A,M)\xla{}0
  \end{align*}
  after observing the fact that $diag{\rm
  Hom}_k(k_\bullet,C_\bullet(A,M))\cong C_\bullet(A,M)$ as cyclic
  $k$-modules.  Now splice the sequences $diag{\rm
  Hom}_k(\xi,C_\bullet(A,M))$ and $\nu$ to get a class in ${\rm
  Ext}^{p+q}_\Lambda(diag{\rm
  Hom}_k(C_\bullet(C,M),C_\bullet(A,M)),k_\bullet^\vee)$.  However, we
  also have a morphism of cyclic $k$-modules $\alpha_\bullet$ 
  constructed in Proposition~\ref{UniversalClass}.  The result follows 
  from the corresponding morphism of Ext-modules 
  \begin{align*}
    {\rm Ext}^{p+q}_\Lambda(\alpha_\bullet,k_\bullet^\vee)\colon
    {\rm  Ext}^{p+q}_\Lambda(diag{\rm Hom}_k(C_\bullet(C,M),C_\bullet(A,M)),k_\bullet^\vee)\to
    {\rm Ext}^{p+q}_\Lambda(Cyc_\bullet(A),k_\bullet^\vee)
  \end{align*}
  and then observing that $HC^*(A) = {\rm
  Ext}^*_\Lambda(Cyc_\bullet(A),k_\bullet^\vee)$.
\end{proof}

If one wishes to write a formula for this pairing, one can write
\begin{align*}
  \xi\smile\nu := {\rm Ext}^{p+q}_\Lambda(\alpha_\bullet,k_\bullet^\vee)\left(diag{\rm Hom}_k(\xi,C_\bullet(A,M))\circ \nu\right)
\end{align*}
for any $\xi\in HC^p_{\rm Hopf}(C,M)$ and $\nu\in HC^q_{\rm
Hopf}(A,M)$ where $\circ$ denotes the Yoneda product in bivariant
cohomology written in the opposite order, i.e.
\[ \circ\colon {\rm Ext}^p_\Lambda(X_\bullet,Y_\bullet)\otimes{\rm Ext}^q_\Lambda(Y_\bullet,Z_\bullet)\to 
   {\rm Ext}^{p+q}_\Lambda(X_\bullet,Z_\bullet)
\]
However, there is one other pairing in the same setting.  One can
define this second pairing by the formula
\begin{align*}
  \nu\smile\xi := {\rm Ext}^{p+q}_\Lambda(\alpha_\bullet,k_\bullet^\vee)
     \left(diag{\rm Hom}_k(C_\bullet(C,M),\nu)\circ diag{\rm Hom}_k(\xi,k_\bullet^\vee)\right)
\end{align*}
Below, we give an alternative construction for these pairings we gave
above and prove that they actually are the same up to a sign.

\begin{prop}
  Assume that an $H$-module coalgebra $C$ acts on an $H$-module
  algebra $A$ equivariantly.  Then for any $\xi\in HC^p_{\rm
  Hopf}(C,M)$ and $\nu\in HC^q_{\rm Hopf}(A,M)$ one has $\xi\smile\nu =
  (-1)^{pq}\nu\smile\xi$.
\end{prop}

\begin{proof}
  Since the bi-functor $diag {\rm Hom}_k(\ \cdot\ ,\ \cdot\ )$ is
  exact in both variables, one has well-defined morphisms of the form
  \[ diag{\rm Hom}_k(Z_\bullet,\ \cdot\ )\colon 
     {\rm Ext}^p_\Lambda(X_\bullet,Y_\bullet) \to 
     {\rm Ext}^p_\Lambda(diag{\rm Hom}_k(Z_\bullet,X_\bullet), diag{\rm Hom}_k(Z_\bullet,Y_\bullet))
  \]
  and 
  \[ diag{\rm Hom}_k(\ \cdot\ ,Z_\bullet)\colon 
     {\rm Ext}^p_\Lambda(X_\bullet,Y_\bullet) \to 
     {\rm Ext}^p_\Lambda(diag{\rm Hom}_k(Y_\bullet,Z_\bullet), diag{\rm Hom}_k(X_\bullet,Z_\bullet))
  \]
  for any cocyclic modules $X_\bullet$, $Y_\bullet$ and cyclic module
  $Z_\bullet$.  Since $\xi\in HC^p_{\rm Hopf}(C,M) := {\rm
  Ext}^p_\Lambda(k_\bullet,C_\bullet(C,M))$ and $\nu\in HC^q_{\rm
  Hopf}(A,M) := {\rm Ext}^q_\Lambda(C_\bullet(A,M),k_\bullet^\vee)$
  one has well-defined elements
  \[ \zeta_1 := diag{\rm Hom}_k(\xi,C_\bullet(A,M))\circ \nu \in 
     {\rm Ext}^{p+q}_\Lambda(diag{\rm Hom}(C_\bullet(C,M),C_\bullet(A,M)), k_\bullet^\vee)
  \]
  and 
  \[ \zeta_2 := diag{\rm Hom}_k(C_\bullet(C,M),\nu)\circ diag{\rm Hom}_k(\xi,k_\bullet^\vee) \in 
     {\rm Ext}^{p+q}_\Lambda(diag{\rm Hom}(C_\bullet(C,M),C_\bullet(A,M)), k_\bullet^\vee) 
  \]
  Here we use the opposite composition notation as before.  After
  observing the fact that $\nu = diag {\rm Hom}_k(k_\bullet,\nu)$, the
  proof that one has $\zeta_1 = (-1)^{pq}\zeta_2$ in ${\rm
  Ext}^{p+q}_\Lambda(diag{\rm
  Hom}_k(C_\bullet(C,M),C_\bullet(A,M)),k_\bullet^\vee)$, reduces to
  proving that the bifunctor $diag{\rm Hom}_k(\ \cdot\ ,\ \cdot\ )$
  satisfies the following property
  \[ diag{\rm Hom}_k(a_\bullet,Y'_\bullet)\circ diag{\rm Hom}_k(X_\bullet,b_\bullet)
     = diag{\rm Hom}_k(X'_\bullet,b_\bullet)\circ diag{\rm Hom}_k(a_\bullet,Y_\bullet)
  \]
  for any morphism of cocyclic modules $a_\bullet \colon X'_\bullet\to
  X_\bullet$ and morphism of cyclic modules $b_\bullet\colon
  Y_\bullet\to Y'_\bullet$ which is obvious.  The result follows after
  observing
  \[ \xi\smile\nu := {\rm Ext}^{p+q}_\Lambda(\alpha_\bullet,k_\bullet^\vee)(\zeta_1)
     = {\rm Ext}^{p+q}_\Lambda(\alpha_\bullet,k_\bullet^\vee)((-1)^{pq}\zeta_2)
     =: (-1)^{pq}\nu\smile\xi
  \]
  where $\alpha_\bullet$ is the morphism of cyclic modules we
  constructed in Proposition~\ref{Lift}.
\end{proof}

\begin{thm}\label{Agreement}
  The Connes-Moscovici characteristic map $HC^p_{\rm
  Hopf}(H,k_{(\sigma,\delta)})\to HC^p(A)$ defined in \cite[Section
  VIII, Proposition 1]{ConnesMoscovici:HopfCyclicCohomology} agrees
  with the pairing we defined in Theorem~\ref{HopfCupProduct} for
  $C=H$, $M=k_{(\sigma,\delta)}$ and $q=0$.  Here
  $k_{(\sigma,\delta)}$ denotes the 1-dimensional anti-Yetter-Drinfeld
  module of the module pair in involution $(\sigma,\delta)$.
\end{thm}

\begin{proof}
  Since $k_{(\sigma,\delta)}$ is a stable anti-Yetter-Drinfeld module
  (an SAYD module in short), by \cite{Kaygun:BialgebraCyclicK} we know
  that $C_\bullet(H,k_{(\sigma,\delta)})$ is isomorphic to $k\otimes_H
  T_\bullet(H,k_{(\sigma,\delta)})$.  Therefore, one can identify
  $C_\bullet(H,k_{(\sigma,\delta)})$ as the graded $k$-submodule of
  $T_\bullet(H,k_{(\sigma,\delta)})$ which consists of elements the
  form $\sum_i (1\otimes h^1_i\otimes\cdots\otimes h^n_i)$.  The
  Connes-Moscovici characteristic is defined with the help of an
  invariant trace $\tau$ on $A$ which satisfies the following
  condition
  \[ \tau(h(a)a') = \tau(a S(h_{(1)})(a')\delta(h_{(2)})) 
     \quad\text{ and }\quad \tau(h(a)) = \varepsilon(h)\tau(a)
  \]
  for any $a,a'\in A$ and $h\in H$.  This is equivalent to $\tau\in
  {\rm Hom}_k(A\otimes k_{(\sigma,\delta)},k)$ being a cyclic cochain
  in degree $0$ for the cyclic module
  $C_\bullet(A,k_{(\sigma,\delta)})$.  The characteristic map is
  defined as
  \[ \gamma_\bullet(1\otimes h^1\otimes\cdots\otimes h^n)(a_0\otimes\cdots\otimes a^n) 
     = \tau(a_0 h^1(a_1)\cdots h^n(a_n)) 
  \]
  for any $1\otimes h^1\otimes\cdots\otimes h^n\in
  C_n(H,k_{(\sigma,\delta)})$ and $a_0\otimes\cdots\otimes a_n\in
  Cyc_n(A)$.  Now observe that one can write $\gamma_\bullet$ as a
  composition $\gamma_\bullet = \tau\circ\alpha_\bullet$ where
  $\alpha_\bullet$ is defined in the proof of Proposition~\ref{Lift}.
  This means $\gamma_\bullet$ is the morphism
  \[ {\rm Hom}_k(\alpha_\bullet,k)\colon {\rm Hom}_k(diag{\rm Hom}_k(C_\bullet(H,M),C_\bullet(A,M)),k)
     \to {\rm Hom}_k(Cyc_\bullet(A),k)
  \]
  where $M = k_{(\sigma,\delta)}$.  Then ${\rm
  Hom}_k(\alpha_\bullet,k)$ induces the morphism ${\rm
  Ext}^*_\Lambda(\alpha_\bullet,k_\bullet^\vee)$ on cohomology which
  is used in the proof of Theorem~\ref{HopfCupProduct} to define the
  pairing.  The result follows.
\end{proof}

One can extend the pairing we defined above to the equivariant cyclic
cohomology groups as follows:

\begin{thm}\label{EquivariantProduct}
  The equivariant action of an $H$-module coalgebra $C$ on an
  $H$-module algebra $A$ induces a pairing of the form
  \begin{align*}
    \smile\colon HC^p_H(C,M)\otimes HC^q_H(A,M)\to 
     \bigoplus_{r=0}^{p+q} HC^{p+q-r}(A)\otimes {\rm Ext}^r_H(k,k)
  \end{align*}
  for any $p,q\geq 0$.
\end{thm}

\begin{proof}
  First observe that the morphism of cyclic modules $\alpha_\bullet$
  defined in Proposition~\ref{Lift} can also be considered as a
  morphism of cyclic $H$-modules of the form $\alpha_\bullet\colon
  Cyc_\bullet(A)\to diag{\rm Hom}_k(Q_\bullet(C,M),Q_\bullet(A,M))$.
  Here $Cyc_\bullet(A)$ is considered as a trivial $H$-module and
  $diag{\rm Hom}_k(Q_\bullet(C,M),Q_\bullet(A,M))$ has the following
  $H$-module structure
  \[ (h f)(\Psi):= h_{(1)}\cdot f(S(h_{(2)})\cdot\Psi) \]
  for any $f\in diag{\rm Hom}_k(Q_\bullet(C,M),Q_\bullet(A,M))$, $h\in
  H$ and $\Psi\in Q_\bullet(C,M)$.  Given two cohomology classes
  \begin{align*}
    \mu: &\quad 0\xla{}k_\bullet\xla{}U^1_\bullet\xla{}\cdots\xla{}U^p_\bullet\xla{}Q_\bullet(C,M)\xla{}0\\
    \nu: &\quad 0\xla{}Q_\bullet(A,M)\xla{}V^1_\bullet\xla{}\cdots\xla{}V^q_\bullet\xla{}k_\bullet^\vee\xla{}0
  \end{align*}
  in $HC_H^p(C,M):={\rm Ext}^p_{H[\Lambda]}(k_\bullet,Q_\bullet(C,M))$
  and in $HC_H^p(A,M):={\rm
  Ext}^p_{H[\Lambda]}(Q_\bullet(A,M),k_\bullet)$ respectively, we
  consider the exact sequence of cyclic $H$-modules
  \begin{align*}
    diag{\rm Hom}_k(\mu,Q_\bullet(A,M))\colon 0\xla{}diag{\rm Hom}_k(Q_\bullet(C,M),Q_\bullet(A,M))\xla{}
    diag{\rm Hom}_k(U^p_\bullet,Q_\bullet(A,M))\xla{}\\
    \cdots\xla{}diag{\rm Hom}_k(U^1_\bullet,Q_\bullet(A,M))\xla{}
    Q_\bullet(A,M)\xla{}0
  \end{align*}
  We define a class in ${\rm Ext}^{p+q}_{H[\Lambda]}(diag{\rm Hom}_k
  (Q_\bullet(C,M),Q_\bullet(A,M)),k_\bullet^\vee)$ by splicing
  $diag{\rm Hom}_k(\mu,Q_\bullet(A,M))$ and $\nu$ at $Q_\bullet(A,M)$.
  Then we define $\mu\smile\nu\in {\rm Ext}^{p+q}_{H[\Lambda]}
  (Cyc_\bullet(A),k_\bullet^\vee)$ by using the morphism of cyclic
  $H$-modules $\alpha_\bullet$.  However, the $H$-module structure on
  $Cyc_\bullet(A)$ is trivial.  Therefore, using the first spectral
  sequence constructed in \cite[Proposition
  3.5]{Kaygun:BivariantHopf} we obtain
  \[ {\rm Ext}^{p+q}_{H[\Lambda]}(Cyc_\bullet(A),k_\bullet^\vee)
     \cong \bigoplus_{r=0}^{p+q} HC^{p+q-r}(A)\otimes {\rm Hom}_k({\rm Tor}_r^H(k,k),k)
  \]
  However, since $k$ is a field 
  we have ${\rm Ext}^r_H(U,k) \cong {\rm Hom}_k({\rm Tor}^H_r(U,k),k)$
  for any $r\geq 0$ and any left $H$-module $U$, in particular $U=k$.
\end{proof}

\section{Product (co)algebras}

We will change the notation to distinguish Hopf-cyclic cohomology and
equivariant cyclic cohomology of $H$- and $H\otimes H$-module
algebras.  We will use $Q_\bullet(H;A,M)$ to denote the cyclic
$H$-module associated with the $H$-module algebra $A$ with
coefficients in $M$.  Also, we will use $HC^*_{\rm Hopf}(H;A,M)$ to
denote the Hopf-cyclic cohomology of an $H$-module algebra $A$ with
coefficients in $M$.

\begin{prop}
  Let $A$, $A'$ be two $H$-module algebras and let $M$ and $M'$ be two
  $H$-module/comodules.  Then there is an external product structure
  on the equivariant cyclic cohomology groups
  \[ HC_H^p(A,M)\otimes HC_H^q(A',M')\to HC_{H\otimes H}^{p+q}(A\otimes A',M\otimes M') \]
  and Hopf-cyclic cohomology groups
  \[ HC_{\rm Hopf}^p(H;A,M)\otimes HC_{\rm Hopf}^q(H;A',M')
     \to HC_{\rm Hopf}^{p+q}(H\otimes H;A\otimes A',M\otimes M') 
  \]
  for any $p,q\geq 0$.
\end{prop}

\begin{proof}
  Any cohomology class $\nu\in HC_H^p(A,M)$ and $\nu'\in
  HC_H^q(A',M')$ can be represented by two exact sequences of cyclic
  $H$-modules of the form $\nu\colon
  0\xla{}Q_\bullet(A,M)\xla{}X_\bullet^1\xla{}
  \cdots\xla{}X_\bullet^p\xla{}k_\bullet^\vee\xla{}0$ and $\nu'\colon
  0\xla{}Q_\bullet(A',M')\xla{}Y_\bullet^1
  \xla{}\cdots\xla{}Y_\bullet^q\xla{}k_\bullet^\vee\xla{}0$.  Now we
  define the external product $\nu\times\nu'$ by the composition
  \begin{align*}
     0\xla{}diag (Q_\bullet(H;A,M)\otimes_k
      Q_\bullet(H;A',M'))\xla{}diag(X_\bullet^1\otimes_k
      Q_\bullet(H;A',M')) \xla{}\cdots\\ \xla{}diag(X_\bullet^p\otimes_k
      Q_\bullet(H;A',M')) \xla{}
      Y_\bullet^1\xla{}\cdots\xla{}Y_\bullet^q\xla{}k_\bullet^\vee\xla{}0
  \end{align*}
  which is an exact sequence of cyclic $H\otimes H$-modules.  Here
  $diag$ denotes the diagonal cyclic structure.  Therefore,
  $diag(U_\bullet\otimes V_\bullet)$ is a cyclic $H\otimes H$-module
  whenever $U_\bullet$ and $V_\bullet$ are cyclic $H$-modules.  Now
  observe that there is a natural epimorphism from $Q_\bullet(H\otimes
  H;A\otimes A',M\otimes M')$ onto $diag(Q_\bullet(H;A,M)\otimes
  Q_\bullet(H;A',M'))$.  This proves our first assertion.  The proof
  of our second assertion is similar.
\end{proof}

Given two left $H$-comodules $M$ and $M'$, we define
\[ M\Box^H M':=\left\{\sum_i m_i\otimes m'_i\ \right|\ \left. 
           \sum_i m_{i,(-1)}\otimes m_{i,(0)}\otimes m_i' = 
           \sum_i m'_{i,(-1)}\otimes m_i\otimes m'_{i,(0)}
           \right\}
\]

\begin{thm}\label{CocommutativeExternalProduct}
  Let $A$, $A'$ be two $H$-module algebras and let $M$ and $M'$ be two
  $H$-module/comodules.  Assume $H$ is cocommutative Hopf algebra and
  $M\Box^H M'$ is an $H$ submodule of $M\otimes M'$.  One has pairings
  of the form
  \[  HC^p_H(A,M)\otimes HC^q_H(A',M')\to HC^{p+q}_H(A\otimes A',M\Box^H M') \]
  and 
  \[  HC^p_{\rm Hopf}(H;A,M)\otimes HC^q_{\rm Hopf}(H;A',M')\to HC^{p+q}_{\rm Hopf}(H;A\otimes A',M\Box^H M') \]
  for any $p,q\geq 0$.
\end{thm}

Now assume $Z$ is a $H$-comodule coalgebra, i.e. $Z$ is an
$H$-comodule and a coalgebra such that the following compatibility
condition is satisfied
\begin{equation}\label{ComoduleCoalgebra}
   z_{[-1]}\otimes z_{[0](1)}\otimes z_{[0](2)}
   = z_{(1)[-1]}z_{(2)[-1]}\otimes z_{(1)[0]}\otimes z_{(2)[0]}
\end{equation}
where the $H$-comodule structure $\lambda\colon Z\to H\otimes Z$ is
denoted by $z_{[-1]}\otimes z_{[0]}$ and the comultiplication
$\Delta\colon Z\to Z\otimes Z$ is denoted by $z_{(1)}\otimes z_{(2)}$
for any $z\in Z$.

\begin{lem}
  Assume $H$ is commutative.  If $Z$ and $Z'$ are two arbitrary
  $H$-comodule coalgebras then their product $Z\otimes Z'$ has an
  $H$-comodule coalgebra structure.
\end{lem}


\begin{defn}
  For a $H$-comodule coalgebra $Z$ and a stable $H$-module/comodule
  $M$, we define the Hopf-cyclic module $C_\bullet(Z,M)$ associated
  with the pair $(Z,M)$ as follows: on the graded $k$-module level we
  let $C_n(Z,M):= {\rm Hom}_{H\text{-comod}}(Z^{\otimes n+1},M)$ where
  we view $Z^{\otimes n+1}$ as an $H$-comodule via the diagonal
  coaction, i.e.
  \[ (z^0\otimes\cdots\otimes z^n)_{[-1]}\otimes
     (z^0\otimes\cdots\otimes z^n)_{[0]}
     := z^0_{[-1]}\cdots z^n_{[-1]}\otimes (z^0_{[0]}\otimes\cdots\otimes z^n_{[0]})
  \]
  for any $z^i\in Z$.  The cyclic structure morphisms are defined by
  \begin{align*}
    (\partial_0 f)(z^0\otimes\cdots\otimes z^{n-1})
    := & f(z^0_{(1)}\otimes z^0_{(2)}\otimes z^1\otimes\cdots\otimes z^{n-1})\\
    (\sigma_0 f)(z^0\otimes\cdots\otimes z^{n+1})
    := & \varepsilon(z^1) f(z^0\otimes z^2\otimes\cdots\otimes z^{n+1})\\
    (\tau_n f)(z^0\otimes\cdots\otimes z^n)
    := & z^0_{[-1]} f(z^1\otimes\cdots\otimes z^n\otimes z^0_{[0]})
  \end{align*}
  for any $f\in C_n(Z,M)$ and $z^i\in Z$.  Then we set $\partial_j :=
  \tau_{n-1}^j\partial_0\tau_n^{-j}$ and $\sigma_i
  :=\tau_{n+1}^i\sigma_0\tau_n^{-i}$ for $0\leq j\leq n$ and $0\leq
  i\leq n$.  The cyclic cohomology of this cyclic $H$-module will be
  denoted by $HC^*_{\rm Hopf}(Z,M)$.  The cyclic cohomology of the
  dual cocyclic object $C_\bullet(Z,M)^\vee$ will be denoted by
  $HC_{\rm Hopf}^{\vee,*}(Z,M)$.
\end{defn}

\begin{thm}\label{CommutativeExternalProduct}
  Assume $Z$ and $Z'$ are arbitrary $H$-comodule coalgebras, $M$ and
  $M'$ are arbitrary $H$-module/comodules.  If $H$ is commutative and
  $M$ and is a symmetric $H$-module then there is a pairing of the
  form
  \[ HC^{\vee,p}_{\rm Hopf}(Z,M)\otimes HC^{\vee,q}_{\rm Hopf}(Z',M')\to 
     HC^{\vee,p+q}_{\rm Hopf}(Z\otimes Z', M\otimes_H M')
  \]
  for any $p,q\geq 0$.
\end{thm}

\begin{proof}
  There is a well-defined morphism of cocyclic $k$-modules of the form
  \[ *\colon diag (C_\bullet(Z,M)\otimes C_\bullet(Z',M'))\to C_\bullet(Z\otimes Z',M\otimes_H M') \]
  given by the formula
  \[ (f*f')((x^0,y^0)\otimes\cdots\otimes (x^n,y^n))
     := f(x^0\otimes\cdots\otimes x^n)\otimes_H f'(y^0\otimes\cdots\otimes y^n)
  \] 
  for any $n\geq 0$, $x^i\in Z$ and $y^i\in Z'$.  It is easy to see
  that $*$ is a morphism of simplicial modules.  We see that
  \begin{align*}
    (\tau_n f)*(\tau_n f')((x^0,y^0)\otimes\cdots\otimes (x^n,y^n))
     = & (\tau_n f)(x^0\otimes\cdots\otimes x^n)\otimes_H (\tau_n f')(y^0\otimes\cdots\otimes y^n)\\
     = & x^0_{[-1]}f(x^1\otimes\cdots\otimes x^n\otimes x^0_{[-1]})\otimes_H 
         y^0_{[-1]}f(y^1\otimes\cdots\otimes y^n\otimes y^0_{[-1]})\\
     = & x^0_{[-1]}y^0_{[-1]} f(x^1\otimes\cdots\otimes x^n\otimes x^0_{[-1]})\otimes_H 
         f(y^1\otimes\cdots\otimes y^n\otimes y^0_{[-1]})\\
     = & (\tau_n(f*f'))((x^0,y^0)\otimes\cdots\otimes (x^n,y^n))
  \end{align*}
  since $M$ is symmetric, as we wanted to show.  Now take two exact
  sequences $\nu\colon 0\xla{} k_\bullet^\vee\xla{}
  U^1_\bullet\xla{}\cdots\xla{}U^p_\bullet\xla{}C_\bullet(Z,M)\xla{}0$
  and $\nu'\colon 0\xla{} k_\bullet^\vee\xla{}
  V^1_\bullet\xla{}\cdots\xla{}V^q_\bullet\xla{}C_\bullet(Z',M')\xla{}0$
  representing two cyclic cohomology class in $HC^{\vee,p}_{\rm
  Hopf}(Z,M)$ and $HC^{\vee,q}_{\rm Hopf}(Z',M')$ respectively.
  Consider the exact sequence
  \begin{align*}
    0\xla{} k_\bullet^\vee & \xla{} U^1_\bullet \xla{}\cdots\xla{} U^p_\bullet\xla{}
        diag(C_\bullet(Z,M)\otimes V^1_\bullet)\xla{}\cdots\\
        & \xla{}diag(C_\bullet(Z,M)\otimes V^q_\bullet) \xla{}  
              diag(C_\bullet(Z,M)\otimes C_\bullet(Z',M'))\xla{}0
  \end{align*}
  which represents a class in ${\rm Ext}^{p+q}_\Lambda(k_\bullet^\vee,
  diag(C_\bullet(Z,M)\otimes C_\bullet(Z',M')))$.  Using the morphism
  $*$ we defined above we get a class in ${\rm
  Ext}^{p+q}_\Lambda(k_\bullet^\vee, C_\bullet(Z\otimes Z',M\otimes_H
  M')$.  The result follows after observing the fact that
  \[ {\rm Ext}^*_\Lambda(k_\bullet^\vee,X_\bullet) \cong
     HC^{\vee,p}(X_\bullet)
  \]
  for any cyclic $k$-module $X_\bullet$.
\end{proof}

\section{Crossed product (co)algebras}

In this section, we will assume $A$ is an $H$-module algebra and $B$
is an $H$-comodule algebra.  $M$ will denote an arbitrary (left-left)
stable anti-Yetter-Drinfeld (SAYD)
module~\cite{Khalkhali:SaYDModules}, i.e. $M$ satisfies
\[ m_{(-1)}m_{(0)} = m \quad\text{ and }\quad
   (hm)_{(-1)}\otimes (hm)_{(0)} = h_{(1)}m_{(-1)} S^{-1}(h_{(3)})\otimes h_{(2)}m_{(0)} \]
for any $m\in M$ and $h\in H$.

\begin{defn}
  We construct a new algebra $A\rtimes B$ which is defined as
  $A\otimes B$ on the $k$-module level.  The multiplication structure
  is defined by the formula
  \[ (a,b)(a',b'):= (a (b_{(-1)} a'), b_{(0)} b') \]
  for $(a,b),(a',b')\in A\rtimes B$.
\end{defn}

Recall the following definition from \cite{Khalkhali:HopfCyclicHomology}.
\begin{defn}
  The Hopf-cocyclic $k$-module $C_\bullet(B,M)$ associated with the
  pair $(B,M)$ is defined by $C_n(B,M) = {\rm
  Hom}_{H\text{-comod}}(B^{\otimes n+1},M)$ on the graded module level
  for any $n\geq 0$.  This means $f\colon B^{\otimes n+1}\to M$ is in
  $C_n(B,M)$ if and only if
  \[ (f(b^0\otimes\cdots\otimes b^n))_{(-1)}\otimes (f(b^0\otimes\cdots\otimes b^n))_{(0)}
     = b^0_{(-1)}\cdots b^n_{(-1)}\otimes f(b^0_{(0)}\otimes\cdots\otimes b^n_{(0)})
  \]
  for any $b^i\in B$ for $i=0,\ldots,n$.  We let
  \begin{align*}
    (\partial_0 f)(b^0\otimes\cdots\otimes b^{n+1})
    := & f(b^0 b^1\otimes b^2\otimes\cdots\otimes b^{n+1})\\
    (\sigma_0 f)(b^0\otimes\cdots\otimes b^{n-1})
    := & f(b^0 \otimes 1_B\otimes b^1\otimes\cdots\otimes b^{n+1})\\
    (\tau_n f)(b^0 \otimes\cdots\otimes b^n)
    := & S(b^n_{(-1)}) f(b^n_{(0)}\otimes b^1\otimes\cdots\otimes b^{n-1})
  \end{align*}
  Then we let define $\partial_j := \tau_{n+1}^{-j}\partial_0\tau_n^j$
  and $\sigma_i := \tau_{n-1}^{-j}\sigma_0\tau_n^j$ for $0\leq j\leq
  n+1$ and $0\leq i\leq n$.  Note that since $M$ is stable,
  $\tau_n^{n+1} = id$ for any $n\geq 0$.
\end{defn}

\begin{prop}\label{CrossedProduct}
  Let us define
  \begin{align*}
    \beta_n( & (a_0,b^0) \otimes\cdots\otimes (a_n,b^n))(f)\nonumber\\
     & := a_0\otimes b^0_{(-n)}a_1\otimes b^0_{(-n+1)}b^1_{(-n+1)}a_2\otimes
            \cdots\otimes b^0_{(-1)}\cdots b^{n-1}_{(-1)} a_n
            \otimes f(b^0_{(0)}\otimes\cdots\otimes b^{n-1}_{(0)}\otimes b^n)\nonumber
  \end{align*}
  for any $n\geq 0$, $(a_i,b^i)\in A\rtimes B$ and $f\in C_n(B,M)$.
  Then $\beta_\bullet$ defines a morphism of cyclic modules of the
  form
  \[ \beta_\bullet\colon Cyc_\bullet(A\rtimes B)\to diag{\rm Hom}_k(C_\bullet(B,M),C_\bullet(A,M)) \]
\end{prop}

\begin{proof}
  We see that 
  \begin{align*}
    \beta_n(\partial_0 & ((a_0,b^0)\otimes\cdots\otimes (a_n,b^n)))(f)\nonumber\\
     = & \beta_n((a_0 b^0_{(-1)}(a_1), b^0_{(0)} b^1)\otimes(a_2,b^2)\cdots\otimes (a_n,b^n))\nonumber\\
     = & a_0 b^0_{(-n)}(a_1)\otimes b^0_{(-n+1)}b^1_{(-n+1)}a_2\otimes\cdots\otimes b^0_{(-1)}\cdots b^{n-1}_{(-1)}a_n
         \otimes f(b^0_{(0)}b^1_{(0)}\otimes b^2_{(0)}\otimes\cdots\otimes b^{n-1}_{(0)}\otimes b^n)\\
     = & \partial_0(a_0\otimes b^0_{(-n)}a_1\otimes b^0_{(-n+1)}b^1_{(-n+1)}a_2\otimes
            \cdots\otimes b^0_{(-1)}\cdots b^{n-1}_{(-1)} a_n)\otimes (\partial_0 f)
            (b^0_{(0)}\otimes\cdots\otimes b^{n-1}_{(0)}\otimes b^n)\nonumber\\
    \beta_n(\sigma_0 & ((a_0,b^0)\otimes\cdots\otimes (a_n,b^n)))(f)\nonumber\\
     = & \beta_n((a_0,b^0)\otimes(1,1)\otimes\cdots\otimes (a_n,b^n)))(f)\nonumber\\
     = & a_0\otimes 1\otimes b^0_{(-n)}a_1\otimes b^0_{(-n+1)}b^1_{(-n+1)}a_2\otimes
            \cdots\otimes b^0_{(-1)}\cdots b^{n-1}_{(-1)} a_n\otimes 
            f(b^0_{(0)}\otimes 1\otimes b^1_{(0)}\otimes\cdots\otimes b^{n-1}_{(0)}\otimes b^n)\\
     = & \sigma_0(a_0\otimes b^0_{(-n)}a_1\otimes b^0_{(-n+1)}b^1_{(-n+1)}a_2\otimes
            \cdots\otimes b^0_{(-1)}\cdots b^{n-1}_{(-1)} a_n)
         \otimes (\sigma_0 f)(b^0_{(0)}\otimes\cdots\otimes b^{n-1}_{(0)}\otimes b^n)\nonumber
  \end{align*}
  and finally we obtain \small
  \begin{align*}
    \beta_n(\tau_n( & (a_0,b^0) \otimes\cdots\otimes(a_n,b^n)))(f)\nonumber\\
    = & \beta_n((a_n,b^n)\otimes (a_0,b^0)\otimes\cdots\otimes(a_{n-1},b^{n-1}))\\
    = & a_n\otimes b^n_{(-n)} a_0\otimes b^n_{(-n+1)} b^0_{(-n+1)}a_1\otimes\cdots\otimes
                   b^n_{(-1)}b^0_{(-1)}\cdots b^{n-2}_{(-1)}a_{n-1}
           \otimes f(b^n_{(0)}\otimes b^0_{(0)}\otimes\cdots\otimes b^{n-2}_{(0)}\otimes b^{n-1})\nonumber\\
    = & S^{-1}(b^n_{(-2)})a_n\otimes a_0\otimes b^0_{(-n+1)}a_1\otimes\cdots\otimes 
               b^{0}_{(-1)}\cdots b^{n-2}_{(-1)}a_{n-1}
           \otimes  S(b^n_{(-1)})f(b^n_{(0)}\otimes b^0_{(0)}\otimes\cdots\otimes b^{n-2}_{(0)}\otimes b^{n-1})\nonumber
  \end{align*}
  \normalsize where the last equality follows from the fact that
  $C_\bullet(A,M)$ is a trivial $H$-module with respect to the
  diagonal $H$-action.  Then using the fact that $M$ is a SAYD module
  and $f\in {\rm Hom}_{H\text{-comod}}(B^{\otimes n+1},M)$ we get
  \begin{align*}
    \beta_n(\tau_n(&(a_0,b^0)\otimes\cdots\otimes(a_n,b^n)))(f)\nonumber\\
    = & S^{-1}(b^n_{(-1)(1)})S^{-1}(b^n_{(0)(-1)}b^0_{(0)(-1)}\cdots b^{n-2}_{(0)(-1)}b^{n-1}_{(-1)}) b^n_{(-1)(3)}
               b^0_{(-1)}\cdots b^{n-1}_{(-1)}a_n\nonumber\\
      &\quad  \otimes a_0\otimes b^0_{(-n)}a_1\otimes\cdots\otimes 
               b^{0}_{(-2)}\cdots b^{n-2}_{(-2)}a_{n-1}\otimes S(b^n_{(-1)(2)})f(b^n_{(0)(0)}\otimes 
               b^0_{(0)(0)}\otimes\cdots\otimes b^{n-2}_{(0)(0)}\otimes b^{n-1}_{(0)})\\
    = & \tau_n\left(\beta_n((a_0,b_0)\otimes\cdots\otimes(a_n,b_n))(\tau_n f)\right)\nonumber
  \end{align*}
  as we wanted to show.
\end{proof}

\begin{thm}\label{CupProductCrossedProduct}
  Assume $M$ is a SAYD module over $H$.  If $A$ is an $H$-module
  algebra and $B$ is an $H$-comodule algebra then there is a pairing
  of the form
  \[ \smile\colon HC^p_{\rm Hopf}(A,M)\otimes HC^q_{\rm Hopf}(B,M)
     \to HC^{p+q}(A\rtimes B)
  \]
  for any $p,q\geq 0$.
\end{thm}

\begin{proof}
  Assume we have two exact sequences
  \begin{align*}
    \nu\colon & 0\xla{} k_\bullet\xla{}V^1_\bullet\xla{}\cdots\xla{}V^p_\bullet\xla{}C_\bullet(B,M)\xla{}0\\
    \nu'\colon & 0\xla{} C_\bullet(A,M)\xla{}W^1_\bullet\xla{}\cdots\xla{}W^q_\bullet\xla{}k_\bullet^\vee\xla{}0
  \end{align*}
  representing two cyclic classes $\nu\in HC^p_{\rm Hopf}(A,M)$ and
  $\nu'\in HC^q_{\rm Hopf}(B,M)$.  We construct
  \begin{align*}
    diag{\rm Hom}_k(\nu,C_\bullet(A,M))\colon 0\xla{}diag{\rm Hom}_k(C_\bullet(C,M),C_\bullet(A,M))\xla{}
    diag{\rm Hom}_k(V^p_\bullet,C_\bullet(A,M))\xla{}\\
    \cdots\xla{}diag{\rm Hom}_k(V^1_\bullet,C_\bullet(A,M))\xla{}C_\bullet(A,M)\xla{}0
  \end{align*}
  and then splice it with $\nu'$ at $C_\bullet(A,M)$ to get a class in
  $HC^{p+q}(diag{\rm Hom}_k(C_\bullet(B,M),C_\bullet(A,M)))$.  Now we
  use the morphism $\beta_\bullet$ of cyclic modules we constructed in
  Proposition~\ref{CrossedProduct} to get a class in
  $HC^{p+q}(A\rtimes B)$.
\end{proof}

\begin{defn}
  Given an $H$-module coalgebra $C$ and an $H$-comodule coalgebra $Z$, we
  define the crossed product coalgebra $Z\ltimes C$ as follows: we let
  $Z\ltimes C := Z\otimes C$ on the $k$-module level.  The counit on $Z\ltimes
  C$ is the tensor product of counits on $C$ and $Z$ respectively.  The
  comultiplication structure is defined by the formula
  \[ (z,c)_{(1)}\otimes(z,c)_{(2)} := (z_{(1)},z_{(2)[-1]}c_{(1)})\otimes (z_{(2)[0]},c_{(2)}) \]
  for any $(c,z)\in Z\ltimes C$.  For coassociativity, one must have
  \begin{align*}
    (z,c)_{(1)(1)}\otimes(z,c)_{(1)(2)}\otimes(z,c)_{(2)} 
    = & (z_{(1)},z_{(2)[-1]}c_{(1)})\otimes (z_{(2)[0]},c_{(2)})_{(1)}\otimes (z_{(2)[0]},c_{(2)})_{(2)}\\
    = & (z_{(1)},z_{(2)[-1]}c_{(1)})\otimes (z_{(2)[0](1)},z_{(2)[0](2)[-1]}c_{(2)})\otimes
          (z_{(2)[0](2)[0]},c_{(3)})
  \end{align*}
  The compatibility conditions for comodule coalgebras give in
  Equation~(\ref{ComoduleCoalgebra}) imply
  \[ z_{[-1]}\otimes z_{[0](1)}\otimes z_{[0](2)[-1]}\otimes z_{[0](2)[0]}
   = z_{(1)[-1]}z_{(2)[-1]}\otimes z_{(1)[0]}\otimes z_{(2)[-1]}\otimes z_{(2)[0]}
  \]
  This in turn yields
  \begin{align*}
    (z,c)_{(1)(1)}\otimes(z,c)_{(1)(2)}\otimes(z,c)_{(2)} 
    = & 
        (z_{(1)},z_{(2)[-1]}z_{(3)[-1]}c_{(1)})\otimes (z_{(2)[0]},z_{(3)[-1]}c_{(2)})\otimes
          (z_{(3)[0]},c_{(3)})\\
    = & (z,c)_{(1)}\otimes (z,c)_{(2)(1)}\otimes (z,c)_{(2)(2)}
  \end{align*}
  for any $z\in Z$ and $c\in C$ as we wanted to show.
\end{defn}

\begin{prop}\label{CocrossedProduct}
  Assume $M$ is an SAYD module.  Let us define $\xi_n$ by the formula \small
  \begin{align*}
    \xi_n((z^0,c^0) & \otimes\cdots\otimes(z^n,c^n))(f)\\
     := & S^{-1}(z^0_{[-1]}\cdots z^n_{[-1]})c^0\otimes S^{-1}(z^1_{[-2]}\cdots z^n_{[-2]})c^1\otimes\cdots\otimes
          S^{-1}(z^n_{[-n-1]})c^n\otimes
          f(z^0_{[0]}\otimes\cdots\otimes z^n_{[0]})\\
      = & S^{-1}(z^0_{[-1]}\cdots z^{n-1}_{[-1]})c^0\otimes S^{-1}(z^1_{[-2]}\cdots z^{n-1}_{[-2]})c^1\otimes\cdots\otimes
          S^{-1}(z^{n-1}_{[-n]})c^{n-1}\otimes c^n\otimes z^n_{[-1]} 
          f(z^0_{[0]}\otimes\cdots\otimes z^n_{[0]})
  \end{align*}
  \normalsize for any $n\geq 0$, $c^i\in C$, $z^i\in C$ and $f\in
  C_n(Z,M)$.  Then $\xi_\bullet$ is a morphism of cocyclic modules of
  the form
  \[ \xi_\bullet\colon Cyc_\bullet(Z\ltimes C)\to diag{\rm Hom}_k(C_\bullet(Z,M),C_\bullet(C,M)) \]
\end{prop}

\begin{proof}
  We will prove that $\xi_\bullet$ and the cyclic operators are
  compatible but we will leave the verification of the fact that
  $\xi_\bullet$ is compatible with the face and degeneracy maps to the
  reader.  We also observe \small
  \begin{align*}
    \xi_n( & \tau_n((z^0,c^0)\otimes\cdots\otimes(z^n,c^n)))(f)\\
      = & \xi_n((z^1,c^1)\otimes\cdots\otimes (z^n,c^n)\otimes (z^0,c^0))(f)\\
      = & S^{-1}(z^1_{[-1]}\cdots z^n_{[-1]})c^1\otimes S^{-1}(z^2_{[-2]}\cdots z^n_{[-2]})c^2\otimes\cdots\otimes
          S^{-1}(z^n_{[-n]})c^n\otimes c^0\otimes 
          z^0_{[-1]} f(z^1_{[0]}\otimes\cdots\otimes z^n_{[0]}\otimes z^0_{[0]})\\
      = & S^{-1}(z^1_{[-2]}\cdots z^n_{[-2]})c^1\otimes S^{-1}(z^2_{[-3]}\cdots z^n_{[-3]})c^2
          \otimes\cdots\otimes S^{-1}(z^n_{[-n-1]})c^n\otimes\\
        & \quad  z^0_{[0][-4]} z^1_{[0][-1]}\cdots z^n_{[0][-1]} z^0_{[0][-1]} S^{-1}(z^0_{[0][-2]}) 
          S^{-1}(z^0_{[-1]}z^1_{[-1]}\cdots z^n_{[-1]})c^0\otimes 
          z^0_{[0][-3]} f(z^1_{[0][0]}\otimes\cdots\otimes z^n_{[0][0]}\otimes z^0_{[0][0]})\\
      = & \tau_n(\xi_n((z^0,c^0)\otimes\cdots\otimes (z^n,c^n))(\tau_n f))
  \end{align*}
  \normalsize
  for any $n\geq 0$, $z^i\in Z$, $c^i\in C$ and $f\in C_n(Z,M)$.
\end{proof}

\begin{thm}\label{CrossedProductCoalgebra}
  Fix an SAYD module $M$ over $H$.  For an $H$-module coalgebra $C$
  and an $H$-comodule coalgebra $Z$ one has a pairing of the form
  \[ HC^p(Z\ltimes C)\otimes HC_{\rm Hopf}^{\vee,q}(Z,M)\to HC_{\rm Hopf}^{p+q}(C,M) 
  \]
  for any $p,q\geq 0$ where $HC_{\rm Hopf}^{\vee,*}(Z,M)$ denotes the
  cyclic cohomology of the dual (cocyclic) module of the cyclic
  $k$-module $C_\bullet(Z,M)$.
\end{thm}

\begin{proof}
  Any class $[\nu]$ in $HC^{\vee,q}_{\rm Hopf}(Z,M)$ represented by an
  exact sequence of the form $\nu\colon
  0\xla{}k_\bullet^\vee\xla{}X^1_\bullet\xla{}
  \cdots\xla{}X^q_\bullet\xla{}C_\bullet(Z,M)\xla{}0$.  Now one can
  construct a class $diag{\rm Hom}_k(\nu,C_\bullet(C,M))$ in the
  bivariant cyclic cohomology group ${\rm Ext}^q_\Lambda 
  (diag {\rm Hom}_k(C_\bullet(Z,M),C_\bullet(C,M)), C_\bullet(C,M))$
  via
  \begin{align*}
    0\xla{}diag{\rm Hom}_k(C_\bullet(Z,M),C_\bullet(C,M))
      \xla{}diag{\rm Hom}_k & (X^q_\bullet,C_\bullet(C,M))\xla{}\cdots\\
      & \xla{}diag{\rm Hom}_k(X^1_\bullet,C_\bullet(C,M))\xla{}
            C_\bullet(C,M) \xla{}0
  \end{align*}
  By using $\xi_\bullet$ we defined in
  Proposition~\ref{CocrossedProduct} we obtain a class in ${\rm
    Ext}^q_\Lambda(Cyc_\bullet(Z\ltimes C),C_\bullet(C,M))$.  In other
  words, we have a morphism of graded modules of the form
  \[ HC_{\rm Hopf}^{\vee,*}(Z,M)\to 
     {\rm Ext}^*_\Lambda(Cyc_\bullet(Z\ltimes C),C_\bullet(C,M))
  \]
  Now pairing with the cyclic cohomology of the crossed product
  coalgebra $HC^p(Z\ltimes C) := {\rm
  Ext}^p_\Lambda(k_\bullet,Cyc_\bullet(Z\ltimes C))$ and using the
  Yoneda composition we get the desired cup product.
\end{proof}

\section{Comparison theorem}\label{ComparisonTheorem}

The first examples of pairings of the form~\ref{HopfCupProduct}
between Hopf-cyclic cohomology of a module algebra $A$ and the
comodule algebra $C=H$ is the Connes-Moscovici characteristic map (for
only $q=0$) \cite{ConnesMoscovici:HopfCyclicCohomology} and its
extension to differential graded setting by Gorokhovsky
\cite{Gorokhovsky:SecondaryCharacteristicClasses} for $q=0,1$ with
periodic cyclic cohomology but only for the 1-dimensional coefficient
module $k_{(\sigma,\delta)}$ coming from a modular pair in involution.
In \cite{Khalkhali:CupProducts}, Khalkhali and Rangipour defined two
cup products for arbitrary module algebras and coalgebras, and their
``cup product of the second kind'' agrees with Gorokhovsky's extended
characteristic map, and therefore with Connes-Moscovici characteristic
map as well, when the coefficient module is $k_{(\sigma,\delta)}$.
There are yet other ways of defining cup products in the context of
Hopf equivariant Cuntz-Quillen formalism due to Crainic
\cite{Crainic:CyclicCohomologyOfHopfAlgebras} and Nikonov-Sharygin
\cite{Sharygin:CupProducts} where the former constructs the cup
product only for $C=H$ and for the 1-dimensional SAYD module
$k_{(\sigma,\delta)}$ while the latter generalizes the construction to
arbitrary $H$-module coalgebras and arbitrary SAYD coefficient
modules.

Our approach to products is different than all of the approaches we
enumerate above in that all the cup products mentioned above use the
theory of abstract cycles and closed graded traces (see \cite[p.
183]{Connes:Book}, \cite[p. 74]{Loday:CyclicHomology} or
\cite{Khalkhali:CupProducts}), and/or the homotopy category of
(special) towers of super complexes.  Thanks to Quillen
\cite{Quillen:CyclicHomologyType} we know that the category of
(special) towers of super complexes is homotopy equivalent to the
category of mixed complexes and the category of $S$-modules.  These
homotopy/derived categories are different than the derived category of
(co)cyclic modules as we observed in the Introduction.  So far, we
explicitly used the derived category of (co)cyclic $k$-modules or
$H$-modules depending on whether we needed the bivariant Hopf or
equivariant bivariant cyclic cohomology.  An alternative approach
would be to develop a similar theory by using the derived category of
mixed complexes.  Recall that in the construction of various products,
we heavily relied on the exact bi-functor $diag{\rm Hom}_k(\ \cdot\ ,\
\cdot\ )$ which takes a pair of cocyclic and cyclic $k$-modules as an
input and which produces a cyclic $k$-module.  If we were to construct
similar products using the derived category of mixed complexes, we
need to mimic the same construction.

One can view the category of mixed complexes as the category of
differential graded $\C{M}$-modules where $\C{M}$ is the free graded
symmetric algebra of the graded vector space $k$ where the single
generator is assumed to have degree $1$.  We will suggestively use $B\in
\C{M}$ to denote this generator.  Then the derived category of
differential graded $\C{M}$-modules is the homotopy category of mixed
complexes~\cite{JonesKassel:BivariantCyclicTheory}.

\begin{lem}[\cite{Nistor:BivariantChernConnes} {Proposition 1.5}]\label{Nistor}
  The functor $\C{B}_*\colon \rmod{\Lambda}\to \dglmod{\C{M}}$ which
  sends a cyclic module to its mixed complex is an exact functor.
  Therefore it induces natural morphisms of derived bifunctors
  \[ {\rm Ext}^*_\Lambda(\ \cdot\ ,\ \cdot\ )
     \to {\rm\bf Ext}^*_\C{M}(\ \cdot\ , \ \cdot \ )
  \]
  where ${\rm\bf Ext}_\C{M}$ stands for the derived functor of the
  Hom-bifunctor of the category of differential graded
  $\C{M}$-modules.
\end{lem}

A bi-differential graded $k$-module $(X_{*,*};b_1,b_2)$ is called a
mixed double complex if there exists differentials $B_1$ and $B_2$ of
degree $1$ such that $[b_i,b_j]=[b_i,B_j]=[B_i,B_j]=0$ whenever $i\neq
j$ and $b_i B_i + B_i b_i = 0$ for $i=1,2$.  For any mixed double
complex $(X_{*,*};b_1,b_2,B_1,B_2)$ we will use $Tot_*(X_{*,*})$ to
denote the mixed complex over the total complex of the bi-differential
graded module $X_{*,*}$ with the degree $1$ differential $B = B_1 +
B_2$.

\begin{lem}\label{GetzlerJones}
  Assume $X_\bullet$ is an arbitrary cocyclic $k$-module and
  $Y_\bullet$ is an arbitrary cyclic $k$-module.  Then in the derived
  category of mixed complexes there is an isomorphism of the form
  \[ \eta_*\colon  Tot_* {\rm Hom}_k(\C{B}_*(X_\bullet),\C{B}_*(Y_\bullet))
     \to
     \C{B}_*(diag{\rm Hom}_k(X_\bullet,Y_\bullet)) 
  \]
\end{lem}

\begin{proof}
  We know that ${\rm Hom}_k(X_\bullet,Y_\bullet)$ is a bi-cyclic
  $k$-module.  If this bi-cyclic module was a product of two cyclic
  modules then we could have used
  \cite{HoodJones:ProductCyclicCohomology}.  Unless $X_n$ is finite
  dimensional for any $n\geq 0$, this is not the case.  But, we can
  still use \cite[Theorem
  3.1]{GetzlerJones:CyclicHomologyOfCrossedProductAlgebras} with
  $T=id$ to get the desired isomorphism $\eta_*$.
\end{proof}

\begin{lem}~\label{Machinery}
  Let $M$ be an arbitrary $H$-module/comodule.  Assume $A$ is an
  $H$-module algebra, $B$ is an $H$-comodule algebra and $C$ is a
  $H$-module coalgebra.  Then there are morphisms of mixed complexes
  of the form \small
  \begin{align*}
    Tot_* {\rm Hom}_k(\C{B}_*(C_\bullet(C,M)),\C{B}_*(C_\bullet(A,M)))
    \xra{\eta_*} \C{B}_*(diag{\rm Hom}_k(C_\bullet(C,M),C_\bullet(A,M)))
    \xla{\C{B}_*(\alpha_\bullet)} \C{B}_*(Cyc_\bullet(A))\\
    Tot_* {\rm Hom}_k(\C{B}_*(C_\bullet(B,M)),\C{B}_*(C_\bullet(A,M)))
    \xra{\eta_*} \C{B}_*(diag{\rm Hom}_k(C_\bullet(B,M),C_\bullet(A,M)))
    \xla{\C{B}_*(\beta_\bullet)} \C{B}_*(Cyc_\bullet(A\rtimes B))
  \end{align*} \normalsize
  where on the second row, we require $M$ to be SAYD as well.
\end{lem}

\begin{proof}
  Left-ward arrows come from applying the functor $\C{B}_*$ to
  $\alpha_\bullet$ we constructed in Proposition~\ref{Lift} and
  Proposition~\ref{UniversalClass}, and to $\beta_\bullet$ which is
  constructed in Proposition~\ref{CrossedProduct}.  The right-ward
  arrows come from Lemma~\ref{GetzlerJones}.
\end{proof}

\begin{thm}\label{AnalogousPairings}
  One can define pairings analogous to the pairings we defined in
  Theorem~\ref{HopfCupProduct}, Theorem~\ref{EquivariantProduct} and
  Theorem~\ref{CocommutativeExternalProduct} by using the derived
  category of mixed complexes instead of the derived category of
  cyclic modules.  However, these pairings defined in the derived
  category of mixed complexes agree with those defined in the derived
  category of cyclic modules.
\end{thm}

\begin{proof}
  The fact that $\eta_*$ is an isomorphism in the derived category of
  mixed complexes allows us to use $Tot_*{\rm Hom}_k(\C{B}_*(\ \cdot\
  ),\C{B}_*(\ \cdot\ ))$ as a replacement for $diag{\rm Hom}_k(\
  \cdot\ ,\ \cdot\ )$ in the category of mixed complexes.  Then
  Lemma~\ref{Nistor} gives us a comparison map between these pairings.
  However, since we are comparing cyclic cohomology of cyclic modules
  computed via the cyclic double complex and the $(b,B)$-complex, we
  see that these groups are the same.
\end{proof}




\end{document}